\documentclass[12pt,leqno,a4paper]{article}
\usepackage{amssymb}
\usepackage{amsmath}
\usepackage{amsthm}
\usepackage{color}
\usepackage[pdfborder={0 0 0}]{hyperref}
\usepackage[shortlabels]{enumitem}
\usepackage{mathrsfs}							
\usepackage{bbm}									
\usepackage{esint}								
\DeclareMathOperator*{\argmax}{arg\,max}
\newtheorem{theorem}{Theorem}

\newtheorem{corollary}[theorem]{Corollary}

\theoremstyle{remark}

\theoremstyle{definition}
\newtheorem{remark}[theorem]{Remark}

\numberwithin{equation}{section}
\numberwithin{theorem}{section}

\newcommand{\R}{\mathbb{R}}

\newcommand{\cC}{{\mathcal C}}

\DeclareMathOperator{\sign}{sign}

\DeclareMathOperator{\divv}{div}

\begin{document}
\title{A variational problem for the maximization of energy dissipation during turbulent fluid mixing in Boussinesq approximation}
\author{ J\'ozsef J. Kolumb\'an\footnote{Department of Analysis and Operations Research, Institute of Mathematics, Budapest University of Technology and Economics, Műegyetem rkp. 3., H-1111 Budapest, Hungary,  and
HUN-REN Alfr\'ed R\'enyi Institute of Mathematics, 1053 Budapest, Re\'altanoda utca 13-15, Hungary,
jkolumban@math.bme.hu, Corresponding author}\quad
and Marietta Oroszki \footnote{Department of Analysis and Operations Research, Institute of Mathematics, Budapest University of Technology and Economics, Műegyetem rkp. 3., H-1111 Budapest, Hungary, 
mariettaoroszki@gmail.com}}
\date{}
\maketitle

\begin{abstract}

We study the Boussinesq approximation of the Rayleigh-Taylor instability with local energy inequality, within the convex integration plus variational selection criterion framework of \cite{RTE}. We show that, under an appropriate rescaling, the energy dissipation functionals obtained in said paper $\Gamma$-converge, and that their limit has a unique maximizer, corresponding to the profile also obtained in \cite{GK_Boussinesq}.

\end{abstract}

\section{Introduction and main results}

The Boussinesq approximation is often used in fluid mechanics to describe the evolution of inhomogeneous fluid flow, when the density difference is small. In particular, we consider two incompressible fluids with homogeneous densities $0<\rho_-<\rho_+$ under the influence of a gravitational force. The system is described by the incompressible 
 Euler equations in Boussinesq approximation
\begin{align}\label{eq:bou}
\begin{split}
\partial_t v +\divv (v\otimes v) +\nabla p&=-\rho gA e_n,\\
\divv v&=0,\\
\partial_t \rho + \divv (\rho v)&=0,
\end{split}
\end{align}
together with the energy (in)equality
\begin{align}\label{eq:euler_local_energy_inequality}
    \partial_t \left(\frac{1}{2}|v|^2+\rho g A x_n\right)+\divv \left(\left(\frac{1}{2}|v|^2+\rho g A x_n+p\right)v\right)&\leq 0,
\end{align}
on $\Omega\times (0,T)$, where
$\Omega\subset\R^n$ is a periodic channel $\mathbb T^{n-1}\times(-L,L)$, $n=2$ or $n=3$, and  $T>0$ is the final time. The unknowns are the normalized fluid density $\rho:\Omega\times[0,T)\rightarrow\R$ ( i.e. $\rho\in\{\pm 1\}$ a.e.), the fluid velocity $v:\Omega\times[0,T)\rightarrow\R^n$ and the fluid pressure $p:\Omega\times[0,T)\rightarrow\R$. We also use $e_n\in\R^n$ to denote the $n$-th coordinate vector, $g>0$ the gravitational constant, as well as 
\[
A:=\frac{\rho_+-\rho_-}{\rho_++\rho_-}
\]
 the so-called Atwood number, which will satisfy $0<A<<1$. As usual for the incompressible Euler equations, one considers no-penetration boundary conditions of the form
\begin{equation}\label{eq:boundary_condition}
v\cdot \vec{n} =0 \quad \text{on }\partial\Omega\times[0,T),
\end{equation}
where $\vec{n}$ denotes the exterior unit normal of the boundary of $\Omega$. 

We are interested in the unstable initial configuration which gives rise to the so-called Rayleigh-Taylor instability (see Rayleigh \cite{Rayleigh} and Taylor \cite{Taylor}), when the heavier fluid is layered on top of the lighter fluid (while gravity is pointing downwards), i.e.
\begin{align}\label{eq:initial_data}
\rho(x,0)=\sign(x_n),\quad v(x,0)=0,\quad x\in\Omega.
\end{align}

In \cite{RTE}, the first author together with B. Gebhard considered a similar problem for the inhomogeneous Euler equations without Boussinesq approximation (hence for arbitrary Atwood number). They constructed a convex integration framework to prove the existence of infinitely many weak solutions of the system, which obey a local energy dissipation inequality, and weakly converge to a so-called "subsolution" (which can be seen as an averaged flow). Furthermore, the authors also
 provided a variational problem as selection criterion for the horizontally averaged profiles, which maximized total energy dissipation.

Indeed, as already observed by Lax \cite{Lax}, taking weak limits of weak solutions of PDEs can be seen as a deterministic alternative to ensemble averaging. Ever since the pioneering work of De Lellis and Sz\'ekelyhidi Jr. \cite{DeL-Sz-Annals,DeL-Sz-Adm} in applying convex integration to prove non-uniqueness of weak solutions of the hydrodynamical Euler equations, one of the main ensuing questions that emerged in the research community was how to construct the most "physical" subsolutions, i.e. the averaged flows that best reflected observed physical reality. This was of particular interest in the case of modelling turbulent fluid mechanical instabilities, such as the Kelvin-Helmholtz \cite{GKEE,Mengual_Sz_vortex_sheet,Sz-KH} and Rayleigh-Taylor instabilities \cite{GKSz,GK_Boussinesq,RTE} or the Muskat problem \cite{Castro_Cordoba_Faraco,Castro_Faraco_Mengual,Castro_Faraco_Mengual_2, Foerster-Sz,Mengual,Mengual_Sz_vortex_sheet,Noisette-Sz,Sz-Muskat}. Many of these approaches also allowed to deduce meaningful quantitative information about the models, such as the maximal growth rate of the turbulent mixing zone, which corresponded to the growth rates observed in physical experiments and numerical simulations (for instance linear growth rate with respect to time for the Kelvin-Helmholtz instability and Muskat problem, and quadratic growth rate for the Rayleigh-Taylor instability).

In Remark 1.2 and Section 1.5.2 of \cite{RTE} it was noted that if one were to apply the limiting transformation corresponding to the Boussinesq approximation to the averaged density profile associated with the maximal energy dissipation, one would get back the same profile that was obtained by a different selection criterion (of maximizing initial energy dissipation) from \cite{GK_Boussinesq}. However, the maximization of energy dissipation in the Boussinesq limit was not rigorously studied in \cite{RTE}.

The main goal of this paper is to fill this gap. Indeed, even if for any $A>0$, we had $F_A=\argmax J_A$ for appropriate energy dissipation functionals $J_A$, and we knew $J_A \to J_{Bou}$ and $F_A\to F_{Bou}$, we could not conclude a priori $F_{Bou}=\argmax J_{Bou}$, in absence of $\Gamma-$convergence or similar criteria. But here the situation is actually more complicated, because within the results of \cite{RTE}, it was only shown that the density profile $\rho_A$ converges to $\rho_{Bou}$, and $\rho_A$ were related to the maximizers $F_A$ via a Riemann-problem. Furthermore, as we will see in Section \ref{sec:scale}, a priori for the construction from \cite{RTE}, there would actually hold $F_A\to 0$, which would not be meaningful in the Boussinesq case. We will address this issue via a rescaling in said section.


For now, we wish to conclude this introduction by stating our main results, for which we introduce the following notation.

Let us denote $a:[-1,1]\to\R$, $$a(r)=1-r^2.$$
We will show in Section \ref{sec:scale} that the total energy dissipation will be given by $ C(t,g,\rho_\pm) J(F)$ with
\begin{align}\label{eq:func0}
J(F) = \int_{-1}^1 \left[ \frac{d}{dr} \left( \frac{F^2}{a(r)}\right) F'(r) + \frac{1}{2}(F'(r))^2 \right] dr,
\end{align}
and $C(t,g,\rho_\pm)<0$ being a constant that depends on $t$, $g$ and $\rho_\pm$.

The first main result of our paper will be the following variational maximization principle of the energy dissipation, and it will be proven in Section \ref{sec:proof2}.

\begin{theorem}\label{thm:var}
Let $$\mathcal F = \Bigl\{ F:[-1,1]\to\mathbb{R} \,\Big|\, F(-1)=F(1)=0,\; F \;\text{convex}\Bigr\},$$
    then there holds
$$\argmax_{F\in\mathcal F}J = -\frac{1}{3}a(r).$$
\end{theorem}

From the discussion in Section \ref{sec:scale}, one may then deduce the following Corollary on the level of subsolutions/averaged turbulent flows, in the spirit of Proposition 1.4 from \cite{RTE}.

\begin{corollary}\label{cor:lambda_1_3}
Among all one-dimensional subsolutions (which can be seen as the horizontally averaged profiles) with $\bar{\rho}$ given as the entropy solution to a Riemann problem 
\begin{align}\label{eq:general_hyperbolic_conservation_law_prop14}
    \partial_t\bar{\rho}+gAt\partial_{x_n} (F(\bar{\rho}))=0,\quad \bar{\rho}(\cdot,0)=\rho_0,
\end{align}
with $F\in\cC^2([-1,1])$ uniformly convex and $F(\pm 1)=0$, the induced total energy dissipation $E(0)-E(t)$ is maximal for $F=-\frac{1}{3}a$.
\end{corollary}

Finally, in Section \ref{sec:proof} we will also show the following convergence result for the (appropriately rescaled) energy dissipation functionals $J_A$ (which are rigorously defined in Section \ref{sec:scale}) for the inhomogeneous Euler equations obtained from \cite{RTE}, in the Boussinesq approximation limit.

 \begin{theorem}\label{thm:gamma}
 The (appropriately rescaled) energy dissipation functionals $\Gamma-$converge in the Boussinesq limit,
 that is, under the definitions of $J$ and $J_A$ as given in \eqref{eq:func0}, respectively \eqref{eq:scale},
 there holds $$-J_A\xrightarrow{\Gamma(C^0)}-J,\text{ as }A\to 0^+.$$
 \end{theorem}
 
 \begin{remark}[Remarks concerning the convex integration]
 In the current paper we are not interested in carrying out any convex integration "by hand", but rather build on the framework of \cite{RTE}, and apply a Boussinesq approximation limit to said framework (hence we also omitted the precise mathematical definition of subsolutions). However, of course one could carry out a similar convex relaxation framework applied to the Euler equations already in Boussinesq approximation \eqref{eq:bou}, and then a similar maximization approach for the energy dissipation. Indeed, we claim that this would simply be a generalization/adaptation of all the calculations from \cite{RTE} in a similar manner as was done in \cite{GK_Boussinesq}, in the absence of local energy inequality. While this would be very lengthy to carry out in detail, it would be mathematically straightforward. In particular, we strongly conjecture that the obtained relaxation would yield a matrix inequality of the form $\lambda_{\max}(M(z))\leq \frac{2}{n} e$, with
 $$M(z)=\frac{v\otimes v-\rho(m\otimes v+v\otimes m)+m\otimes m}{1-\rho^2}-\sigma.$$
 Note that here $e$ is part of the unknown functions within the subsolution, hence is allowed to oscillate, thus outplaying the difficulties from \cite{GK_Boussinesq}, when one had to prescribe the energy in some appropriate sense.
 
 Since, following Section 3 of \cite{RTE}, the one-dimensional subsolutions considered satisfy $v\equiv 0$, the above matrix naturally reduces to $$\frac{m\otimes m}{1-\rho^2}-\sigma,$$ and thus further adapting the approach of Section 3 of said paper (noting the lack of domain transformation, similarly to \cite{GK_Boussinesq}), the ansatz $m=m_ne_n$ with $m_n=gAt F(\rho)$ would yield that the energy dissipation is given by (up to a constant depending on $g$, $A$ and $t$) precisely the functional $J$ from \eqref{eq:func0}.
 
 However, in the current paper we do not need to rely on the above (lengthy if made rigorous) arguments, we show that the same dissipation functional appears in the $\Gamma-$limit if one applied the Boussinesq approximation after the relaxation of \cite{RTE}. Nonetheless, we once more emphasize that it is of interest to know that the same structure appears naturally within both approaches.
 \end{remark}

The remainder of the paper is structured as follows. In Section \ref{sec:scale} we give rigorous motivation as to why the above maximization problem is indeed the relevant one in Boussinesq approximation, using a careful rescaling of the construction from \cite{RTE}. 
Then in Section \ref{sec:proof} we show that actually there is $\Gamma-$convergence of the energy dissipation functionals under our rescaling in the Boussinesq approximation limit. While this allows us to infer the existence of a maximizer for the limiting problem, in order to also obtain uniqueness, in Section \ref{sec:proof2} we
actually give a direct proof of the maximization, which in part diverges from the similar one used in \cite{RTE}, using some different integral estimates, which turn out to be more efficient and allow for more simplification.

\section{Obtaining the dissipation functional in the Boussinesq limit}\label{sec:scale}

Let us first start with a quick presentation of the heuristics behind the Boussinesq approximation. Afterwards, we will summarize the results of \cite{RTE}, explain the issues when trying to pass to the Boussinesq limit, and provide a rescaling solution, which will serve as preliminaries for our proof of the maximization of the energy dissipation from Sections \ref{sec:proof} and \ref{sec:proof2}.

The Boussinesq approximation for the inhomogeneous Euler equations
\begin{align}\label{eq:euler_equations}
\begin{split}
\partial_t (\tilde{\rho} v) +\divv (\tilde{\rho}v\otimes v) +\nabla \tilde{p}&=-\tilde{\rho} g e_n,\\
\divv v&=0,\\
\partial_t \tilde{\rho} + \divv (\tilde{\rho} v)&=0,
\end{split}
\end{align}
consists of applying the normalization
$$\tilde{\rho}=\frac{\rho_++\rho_-}{2}(1+\rho A),$$
such that $\tilde\rho\in\{\rho_\pm\}$ is equivalent to $\rho\in\{\pm 1\}$,
and then considering the limit $A\to 0^+$, with the added ansatz that if the Atwood number is multiplied by gravity, then it does not vanish in the limit, since there is buoyancy due to the gravity magnifying the otherwise small density difference.

Then, the Cauchy momentum equation from \eqref{eq:euler_equations} becomes
\begin{align*}
(1+\rho A)(\partial_t v +\divv (v\otimes v)) +\nabla p=-\rho gA e_n,
\end{align*}
where the pressures are related by the (rescaled) hydrostatic balance equation
$$\frac{\rho_++\rho_-}{2} p = \tilde p +\frac{\rho_++\rho_-}{2} x_n.$$
Clearly, passing to the limit $A\to 0^+$, but applying the rule of Boussinesq approximation that $g A$ does not vanish (in particular in the force on the right-hand side), yields the first equation of \eqref{eq:bou}. The fact that the transport equation is invariant under normalization is a trivial exercise left to the reader.

\subsection{Summary of the results from \cite{RTE}}

In Section 3.3 of \cite{RTE} it was shown that (up to an arbitrarily small error) the energy dissipation for the inhomogeneous Euler equations (without Boussinesq approximation) is given by
\begin{align}\label{eq:diss}
E_{tot}(t)-E_{tot}(0)=-\frac{g^3t^4}{4} \int_{\rho_-}^{\rho_+} \left[ \frac{d}{d\rho} \left( \frac{F^2}{\alpha(\rho)}\right) F'(\rho) + \frac{1}{2}(F'(\rho))^2 \right] d\rho,
\end{align}
where
\begin{align}
\alpha(\rho)=\frac{(\rho_+-\rho)\rho-\rho_-)}{\rho_++\rho_--\rho}.
\end{align}

Furthermore, as discussed in Remark 1.2 and Section 1.5.2 of \cite{RTE}, the normalization of the density and the passage to the Boussinesq approximation limit is obtained by using the bijection
\begin{align}\label{eq:bij}
\rho(r)=\frac{\rho_++\rho_-}{2}(1+r A),
\end{align}
as we have also mentioned at the beginning of the section.
In particular, the normalized profile corresponding to the optimal averaged density obtained in \cite{RTE} was given by $\bar r_A (x_n,t)$ with
\begin{align}\label{eq:profA}
\bar r_A\left(\frac{Ag t^2}{3}\xi,t\right)=\frac{1}{A}\left(1- \sqrt{\frac{1-A^2}{1+2\xi}}\right),\quad \xi\in\left(\frac{-1}{1+A},\frac{1}{1-A} \right),
\end{align}
from where one easily obtained in the limit $A\to 0^+$ of the right-hand side the approximation
\begin{align*}
\xi+O(A),\quad \xi\in\left(-1+O(A),1+O(A) \right),
\end{align*}
and hence one can define
$$\bar r_{Bou}\left(\frac{Ag t^2}{3}\xi,t\right)=\xi,\quad \xi\in(-1,1),$$
which is exactly the profile appearing in \cite{GK_Boussinesq}.

We also note that the density $\rho_A=\rho(\bar r_A)$ was obtained as the solution of a Riemann problem (see \eqref{eq:rie} later for the precise form) associated with the (unique) maximizer $F_A(\rho)=-\frac{1}{3}\alpha(\rho)$ of the functional
\begin{align}\label{eq:funcA}
\int_{\rho_-}^{\rho_+} \left[ \frac{d}{d\rho} \left( \frac{F^2}{\alpha(\rho)}\right) F'(\rho) + \frac{1}{2}(F'(\rho))^2 \right] d\rho.
\end{align}

\subsection{Limit issues and finding a rescaling to obtain the right functionals}

While it was trivial to pass to the limit in the right-hand side of \eqref{eq:profA} (simply using L'H\^ opital's rule), for the functionals and their maximizers we immediately observe the following issues. On one hand, we have via a simple calculation that 
\begin{align*}
\alpha(\rho(r))=\frac{\rho_++\rho_-}{2} \frac{A^2}{1-r A}(1-r^2),
\end{align*}
which a priori converges to zero as $A\to 0^+$ (and hence we would have $F_A=-\frac{1}{3}\alpha\to 0$ as well). However, if we first rescale by dividing with $\frac{\rho_++\rho_-}{2}A^2$, we would get the desired behavior
$$\tilde\alpha_A(r):=\frac{2}{(\rho_++\rho_-)A^2}\alpha(\rho(r))=\frac{1-r^2}{1-rA}\xrightarrow{C^\infty} a(r)=1-r^2.$$
This gives us the intuition of the correct scaling also on the level of all possible functions $F$.

Furthermore, to understand the effect of the change of variables \eqref{eq:bij} in \eqref{eq:funcA}, we observe that
\begin{align*}\frac{d\rho}{dr}=\frac{\rho_++\rho_-}{2}A,
\end{align*}
and thus, for any differentiable $G:[\rho_-,\rho_+]\to\mathbb R$, setting
$$\tilde G(r):=\frac{2}{(\rho_++\rho_-)A^2}G(\rho(r)),$$ there holds
\begin{align}\label{eq:diffch}
\frac{d}{dr} \tilde G(r)=\frac{d}{d
\rho}G(\rho(r)) \frac{1}{A}.
\end{align}
In particular, note that we would have
$$\widetilde{\left(\frac{F^2}{\alpha}\right)}=\frac{2}{(\rho_++\rho_-)A^2} \frac{F(\rho(r))^2}{\alpha(\rho(r))}=\frac{\left(\tilde F(r)\right)^2}{\tilde\alpha_A(r)}.$$
Note that the above identities also extend to weak derivatives, in the appropriate sense, as long as all involved integrals are well-defined.

Thus, if we apply the above with $G=F$ and then $G=\frac{F^2}{\alpha}$ in \eqref{eq:funcA}, we may rigorously carry out the change of variables corresponding to \eqref{eq:bij} in order to get
\begin{multline*}
\int_{\rho_-}^{\rho_+} \left[ \frac{d}{d\rho} \left( \frac{F^2}{\alpha(\rho)}\right) F'(\rho) + \frac{1}{2}(F'(\rho))^2 \right] d\rho = \\ \int_{-1}^1 \left[ A\frac{d}{dr} \left( \frac{\tilde F^2}{\tilde{\alpha}_A(r)}\right)A\tilde F'(r) + \frac{1}{2}\left(A\tilde F'(r)\right)^2 \right] \frac{\rho_++\rho_-}{2}A dr
\\= \frac{\rho_++\rho_-}{2}A^3 \int_{-1}^1 \left[ \frac{d}{dr} \left( \frac{\tilde F^2}{\tilde{\alpha}_A(r)}\right)\tilde F'(r) + \frac{1}{2}\left(\tilde F'(r)\right)^2 \right] dr.
\end{multline*}
Hence, keeping in mind that $\tilde\alpha_A(r)\to a(r)$ as $A\to 0^+$, this allows us to rigorously define for any $A>0$ the functional
\begin{align}\label{eq:scale}
J_A(\tilde F):=\frac{8(E_{tot}(0)-E_{tot}(t))}{(\rho_++\rho_-)g^3A^3 t^4}=\int_{-1}^1 \left[ \frac{d}{dr} \left( \frac{\tilde F^2}{\tilde{\alpha}_A(r)}\right)\tilde F'(r) + \frac{1}{2}\left(\tilde F'(r)\right)^2 \right] dr,
\end{align}
which still corresponds to the energy dissipation in the case of the inhomogeneous Euler equations without Boussinesq approximation (but in the normalized and rescaled setting). In the next section we will show that 
$$-J_A \xrightarrow{\Gamma(C^0)} -J_{Bou}:=-J\text{ as }A\to 0^+,$$
where $J$ is exactly the functional given in \eqref{eq:func0}.

We conclude this section with two remarks. On one hand, it should be noted that the scaling obtained in \eqref{eq:scale} is in harmony with the construction from \cite{GK_Boussinesq}, where a different relaxation was used, with the selection criterion of maximizing the initial energy dissipation instead of the total dissipation. Nonetheless, as already mentioned, the same limiting optimal density profile (here corresponding to $F_{Bou}=-\frac{1}{3}a$) was obtained in said paper, and the initial energy dissipation also scaled like $-g^3 A^3 t^4$, see e.g. Example 4.1. from \cite{GK_Boussinesq}.

Finally, in order to conclude Corollary \ref{cor:lambda_1_3}, given Theorem \ref{thm:var}, it suffices to note that the Riemann problem \eqref{eq:general_hyperbolic_conservation_law_prop14} follows naturally given the transformations we introduced in this section. Indeed, from Corollary 1.3 of \cite{RTE}, one has that $\rho_A=\rho(\bar r_A)$ satisfies the conservation law
\begin{align}\label{eq:rie}
\partial_t\rho_A+g t\partial_{x_n} (F_A(\rho_A))=0,
\end{align}
with $F_A(\rho)=-\frac{1}{3}\alpha(\rho)$. Using \eqref{eq:bij} and \eqref{eq:diffch}, one gets
\begin{align*}
0&=\frac{\rho_++\rho_-}{2}A\partial_t \bar r_A -\frac{1}{3} g t \frac{d}{d\rho}\alpha(\rho(\bar r_A)) \frac{\rho_++\rho_-}{2}A \partial_{x_n} \bar r\\&=\frac{\rho_++\rho_-}{2}A \left (\partial_t \bar r_A-\frac{1}{3} g t \frac{d}{dr}\tilde\alpha_A(\bar r_A) A \partial_{x_n} \bar r_A\right)=\frac{\rho_++\rho_-}{2}A \left (\partial_t \bar r_A-\frac{1}{3} g A t \partial_{x_n} \tilde\alpha_A(\bar r_A) \right), 
\end{align*}
hence in the Boussinesq limit $\tilde\alpha_A(r)\to a(r)$ and $r_A\to r_{Bou}$, one reobtains
$$\partial_t \bar r_{Bou}-\frac{1}{3} g A t \partial_{x_n} a(\bar r_{Bou})=0\quad \Leftrightarrow\quad  \partial_t \bar r_{Bou}+ g A t \partial_{x_n} (F_{Bou}(\bar r_{Bou}))=0.$$

\section{Proving the $\Gamma-$convergence of the dissipation functionals in the Boussinesq limit}\label{sec:proof}


For the sake of completeness, let us first show that $J$ is in fact well-defined over the whole of $\mathcal F$, although this is a similar argument to that in \cite{RTE}. Since any $ F \in \mathcal F$ is convex and vanishes at the endpoints, it follows that 
$$|F(r)|\lesssim (1-r^2)$$
near the boundary points,
 hence $\frac{F}{a}\in L^\infty(-1,1)\cap W^{1,\infty}_{loc}(-1,1)$ and $\frac{F^2}{a}\in  W^{1,\infty}(-1,1)$. This easily yields that $J(F)$ is well-defined. Similarly one can show that $J_A$ is also well-defined on $\mathcal F$. Now we can prove our result on $\Gamma-$convergence of the functionals.

\begin{proof}[Proof of Theorem \ref{thm:gamma}]
 
\textbf{Step 1.} 
 
To show $\Gamma-$convergence, we will first show locally uniform convergence in $W^{1,\infty}$. Let $M>0$, for $F\in\mathcal F$ such that $\|F\|_{W^{1,\infty}}\leq M$
we have in fact
$$|F(r)|\leq M (1-r^2)$$
near the boundary points $\pm 1$.
We may then further write
\begin{align*}
\frac{d}{dr} \left( \frac{F^2}{\tilde{\alpha}_A(r)}\right)-\frac{d}{dr} \left( \frac{ F^2}{a(r)}\right) = - A  \frac{d}{dr} \left(r \frac{ F^2}{a(r)}\right),
\end{align*}
hence
\begin{align*}
J_A(F)-J(F)=-A\int_{-1}^1 \frac{d}{dr} \left(r \frac{ F^2(r)}{a(r)}\right) F'(r) \, dr.
\end{align*}
Let $G(r):=\frac{ F^2(r)}{a(r)}$, recalling $|F(r)|\leq M (1-r^2)$ near the boundary points $\pm 1$, we may estimate
\begin{align*}
\left|\frac{d}{dr} \left(r \frac{ F^2(r)}{a(r)}\right)\right| \leq \|G\|_{W^{1,\infty}}\lesssim C M ^2,
\end{align*}
for a universal constant $C>0$ independent of $F$.
Thus, we obtain
\begin{align*}
|J_A(F)-J(A)|\leq 2C M^3 A,
\end{align*}
which implies that $J_A \to J$ locally uniformly in $\mathcal F$ with respect to the $W^{1,\infty}$ topology as $A\to 0^+$.

\textbf{Step 2.} 

Let us show that $J$ is continuous with respect to the $C^0$ topology on $\mathcal F$.
Let $F_n \in \mathcal F$ such that $F_n \to F$ uniformly, due to convexity of all $F_n$, one has $F'_n\to F'$ almost everywhere, and the derivatives are uniformly bounded in $L^\infty$. Furthermore,
$$G_n(r):=\frac{F_n(r)^2}{1-r^2}$$ 
are uniformly Lipschitz and converge uniformly to $G(r):=\frac{F(r)^2}{1-r^2}$.

Hence, $G'_n$ weak-$\star$ converge to $G'$ in $L^\infty$. Since we have
\begin{align*}
J(F_n)=\int_{-1}^1 \left(G'_n F'_n + \frac{1}{2} (F'_n)^2 \right) \, dr,
\end{align*}
recalling the almost everywhere convergence and uniform boundedness of $F'_n$, 
one can apply the dominated convergence theorem to obtain $J(F_n)\to J(F)$.

\textbf{Step 3.}

We can now show the liminf (in)equality.
Let now $A\to 0^+$, and $F_A \in \mathcal F$ such that $F_A \to F$ uniformly. Uniform convergence of convex functions on a compact interval implies a uniform Lipschitz bound, hence there exists $M>0$ such that 
$$\|F_A\|_{W^{1,\infty}} \leq M,\ \forall A\in(0,1).$$

From Step 1. we get
$$|J_A(F_A)-J(F_A)|\leq 2 C M^3 A,$$
but from Step 2. we know that
$J(F_A)\to J(F)$, therefore
$$\liminf_{A\to 0^+} -J_A(F_A)=-J(F).$$

\textbf{Step 4.}

For the recovery sequence we can simply take for any given $F\in\mathcal F$, $F_A\equiv F$, so we have
\begin{align*}
\lim_{A\to 0^+}-J_A(F_A)=\lim_{A\to 0^+}-J_A(F)=-J(F),
\end{align*}
once more by Step 1.

This concludes the proof of the $\Gamma-$convergence.

\end{proof}

Therefore, we have made the convergence of the functionals rigorous.
The results of \cite{RTE} and our analysis from Section \ref{sec:scale} yield that, for any $A\in(0,1)$, the unique maximizer of $J_A$ is exactly $-\frac{1}{3}\tilde\alpha_A(r)$, which $C^\infty$-converges to $-\frac{1}{3}a(r)$ as $A\to 0^+$. Thus, by $\Gamma-$convergence of $-J_A$ to $-J$,  $-\frac{1}{3}a(r)$ must be a maximizer of $J$. However, this method does not yield uniqueness. For the latter, we will employ a direct approach in the next section.

\section{Direct proof of the unique maximization property}\label{sec:proof2}

Finally, we are in position to prove Theorem \ref{thm:var}. As we have seen in Section \ref{sec:proof}, the only missing ingredient is the uniqueness of the maximizer, which we will show by direct integral estimates, that actually also give an alternative proof of the existence of the maximizer.

We note that while the $\Gamma-$convergence obtained in Section \ref{sec:proof} is a property of the model and the Boussinesq limit that is of interest in of itself, this direct proof of the maximization result has the advantage of not relying on needing to know that the functionals $J_A$ all have (unique) maxima.

Furthermore, while the proof is based on similar ideas as that of Lemma 3.4 from \cite{RTE}, it is not simply a word-by-word adaptation of said proof, in its second half we diverge from the construction in \cite{RTE}, using more efficient estimates, so that unlike said paper, here we do not need to keep going back and forth between the $H$ and $F$ viewpoints. Hence, our proof is a refinement of the proof of Lemma 3.4 from \cite{RTE}.

\begin{proof}[Proof of Theorem \ref{thm:var}]

We start the proof by using the same equivalent formulation as in \cite{RTE}, by replacing $F(r)=-a(r)H(r)$, to get
\begin{align*}
    J(-aH) &= \int_{-1}^1 \left[ \frac{1}{dr} \left(\frac{a^2H^2}{a}\right)(-a'H-aH')  + \tfrac{1}{2} \bigl(-a'H-aH'\bigr)^2 \right] \, dr \\
    &=  \int_{-1}^1 \left[ -\left(a'H^2+2aHH'\right)(-a'H-aH')  + \tfrac{1}{2} \bigl(-a'H-aH'\bigr)^2 \right] \, dr \\
    &=  \int_{-1}^1 \frac{1}{2}\left[ a^2(1-4H)(H')^2 -aa''H^2(1-2H) \right]\, dr =:\tilde{J}(H).
\end{align*}
Thus, we will show that for any
 $$H\in \mathcal H = \Bigl\{ H\in  L^\infty(-1,1)\cap W^{1,\infty}_{loc}(-1,1)  \,\Big|\, H\geq 0,\; -aH \;\text{convex}\Bigr\}$$ there holds
$$\tilde J (H)\leq \tilde J (1/3),$$
and the inequality is strict if $H\neq 1/3$.

To do so, we continue to follow some of the ideas from \cite{RTE}, but then diverge half-way through the proof, using more efficient direct inequalities, to simplify the second half thereof.

We may split
\begin{multline}\label{eq:est}
    \tilde{J}(H) - \tilde{J}(1/3)  \\= \int_{\{ r : H(r) > \frac{1}{4} \}} \left[ \frac{1}{2} a^2 (1 - 4H)(H')^2 - \frac{1}{2} a a'' H^2 (1 - 2H) +\frac{1}{54} a a''\right] \, dr \\
      + \int_{\{ r : H(r) < \frac{1}{4} \}} \left[ \frac{1}{2} a^2 (1 - 4H)(H')^2 - \frac{1}{2} a a'' H^2 (1 - 2H)+\frac{1}{54} a a'' \right] \, dr. 
\end{multline}

The first term has zero as its upper bound, since 
$1-4H \leq0$, $-aa'' \geq 0$, and the maximum of $H^2 (1 - 2H)$ is realized in  $H=1/3$.

For the second term, let $-1<r_1<r_2<1$, and introduce the following restricted functional:
\begin{align}\label{eq:func}\tilde J_{r_1,r_2} (H)=\int_{r_1}^{r_2} \frac{1}{2}\left[ a^2(1-4H)(H')^2 -aa''H^2(1-2H) \right]\, dr.
\end{align}
We note that $\{ r : H(r) < \frac{1}{4} \}$ can be written as the union of maximally disjoint intervals, for which any $r_0$ of its endpoints yields 
 $H(r_0)=1/4$ or $r_0\in\{\pm 1\}$. In both cases one gets $a(r_0)H(r_0)^k=a(r_0)(1/4)^k,$ for any $k\geq 0$.

Let $(r_1,r_2)\subset\{ r : H(r) < \frac{1}{4} \}$ be such an interval.
For simplicity, let us assume that $aH$ is $C^2$ on $(r_1,r_2)$, the general case can be inferred either by approximation and mollification, or by interpreting $(aH)''$ as a measure and the boundary terms in a trace sense.

By partial integration, we obtain
\begin{multline*}
    \int_{r_1}^{r_2} a^2 (1 - 4H)(H')^2 \, dr = \\-  \int_{r_1}^{r_2} [a^2 H(1-4H)H'' +2aa'(1-4H)HH'-4a^2(H')^2H] \, dr=
    \\-  \int_{r_1}^{r_2} aH(1-4H)\left[2a'H'+aH''\right] \, dr+\int_{r_1}^{r_2} 4a^2(H')^2H \, dr.
\end{multline*}
Now, using the fact that
$$a''H =(aH)'' -2a'H'-aH'',$$
we get
\begin{align*}
    \int_{r_1}^{r_2} aa''H^2(1-2H) \, dr = \int_{r_1}^{r_2} aH(1-2H)\left[(aH)''-2a'H'-aH''\right] \, dr. 
\end{align*}
Hence the value of the restricted functional \eqref{eq:func} becomes
\begin{align*}
     \tilde{J}_{r_1,r_2}(H)&=\frac{1}{2}\int_{r_1}^{r_2} \left[ 4a^2(H')^2H +2aH^2(2a'H'+aH'')\right]\, dr \\
     & \quad - \frac{1}{2} \int_{r_1}^{r_2} a H (1 - 2H) (a H)'' \, dr.
\end{align*}
Thus, using the identity
\begin{align*}
    \frac{d}{dr}(a^2 H^2 H')=2a^2(H')^2H+2aa'H^2H'+a^2H^2 H''
\end{align*}
we may further write the previous as
\begin{align*}
    \tilde{J}_{r_1,r_2}(H) &=  \int_{r_1}^{r_2} \frac{d}{dr}(a^2 H^2 H') \, dr - \frac{1}{2} \int_{r_1}^{r_2} a H (1 - 2H) (a H)'' \, dr\\
    &= \left[ a^2 H^2 H' \right]_{r_1}^{r_2} - \frac{1}{2} \int_{r_1}^{r_2} a H (1 - 2H) (a H)'' \, dr\\
    &\leq - \frac{1}{16} \int_{r_1}^{r_2} a (aH)'' \, dr + \left[ \frac{1}{16} a^2 H' \right]_{r_1}^{r_2} \\
    &= \frac{1}{16} \int_{r_1}^{r_2} a' (aH)' \, dr + \left[ \frac{1}{16} \left( a^2 H' - a (aH)' \right) \right]_{r_1}^{r_2} \\
    &= -\frac{1}{16} \int_{r_1}^{r_2} a'' a H \, dr + \left[ \frac{1}{16} \left( a^2 H' - a (aH)' + a' aH \right) \right]_{r_1}^{r_2} \\
    &\leq - \frac{1}{64} \int_{r_1}^{r_2} a'' a \, dr,
\end{align*}
where we have used that $-(aH)''\leq 0$, and that $H(1 - 2H)$ has a maximal value of $\frac{1}{8}$, since $H \in [0; \frac{1}{4}]$ is maximal for $H = \frac{1}{4}$.

From here we get that
$$\tilde{J}_{r_1,r_2}(H)-\tilde{J}_{r_1,r_2}(1/3)=\left(\frac{1}{54}- \frac{1}{64} \right)\int_{r_1}^{r_2} a'' a \, dr\leq 0.$$
And the above yield 
$\tilde{J}(H)\leq \tilde{J}(1/3)$,
as well as
$\tilde{J}(H) < \tilde{J}(1/3)$ if the set
$\{ r : H(r) < \frac{1}{4} \}$ is not empty. 
In the case when this set is empty, the second term in \eqref{eq:est} is zero, and thus we would once more get the strict inequality if $H\neq 1/3$.

\end{proof}

\section*{Acknowledgements} 

J. J. Kolumb\'an 
was supported by the J\'anos Bolyai Research Scholarship of the Hungarian Academy of Sciences.

\section*{Statements and Declarations}

On behalf of all authors, the corresponding author states that there is no conflict of interest.

\section*{Data availability} We do not analyse or generate any datasets.

\end{document}